\documentclass{article}
\usepackage{amsmath}
\usepackage{amssymb}

\newcommand{\EE}{\mathbb{E}}
\newcommand{\ZZ}{\mathbb{Z}}
\newcommand{\NN}{\mathbb{N}}

\newcommand{\RR}{\mathbb{R}}

\newcommand{\T}{\mathbb{T}}
\newcommand{\bP}{{\bf P}}

\newcommand{\hH}{{\cal H}}
\newcommand{\I}{{\cal I}}

\newcommand{\hLambda}{\widehat{\Lambda}}

\newtheorem{theorem}{Theorem}

\newtheorem{proof}{Proof}

\newtheorem{corollary}{Corollary}

\begin{document}
\begin{center}
{\Large STIT Tessellations -- Ergodic Limit Theorems and Bounds for 
the Speed of Convergence}

\vspace{1cm}

{Servet Mart{\'i}nez}\\
{Departamento Ingenier{\'\i}a Matem\'atica and Centro
Modelamiento Matem\'atico,\\  Universidad de Chile,\\
UMI 2807 CNRS, Casilla 170-3, Correo 3, Santiago, Chile.\\
Email: smartine@dim.uchile.cl} 

\vspace{0.5cm}

{Werner Nagel}\\
{Friedrich-Schiller-Universit\"at Jena,\\
Institut f\"ur Stochastik,\\
Ernst-Abbe-Platz 2,
D-07743 Jena, Germany.\\
Email: werner.nagel@uni-jena.de} 

\end{center}

\begin{abstract}
We consider homogeneous STIT tessellations in the $\ell$-dimensional 
Euclidean space $\RR^\ell$. Based on results for the 
spatial $\beta$-mixing coefficient an 
upper bound for the variance of additive functionals 
of tessellations is derived, using results by Yoshihara and Heinrich. 
Moreover, ergodic theorems are applied to subadditive functionals.
\end{abstract}


\vspace{.5cm}


\section{Introduction}
\label{sec0}

Random tessellations form an important class of models studied in 
stochastic geometry. Besides the well-established Poisson hyperplane 
tessellations and Poisson-Voronoi tessellations, the STIT tessellation 
extends the variety of models that are interesting from a theoretical as 
well as a practical point of view. An essential aspect for such models is 
whether ergodic limit theorems or even CLT's can be shown for certain 
functionals of tessellations. This forms the background for statistical 
applications, as parameter estimation, tests or confidence intervals.

\medskip

For the Poisson-Voronoi tessellations, a seminal paper dealing with these 
aspects is \cite{hein94}. There is also an appropriate concept developed 
for mixing coefficients in space, including sufficient conditions for 
CLT's. A recent review can be found in \cite{hein}.

\medskip

For STIT tessellations, which were introduced in \cite{nw}, a profound 
study of second order properties and the limit behavior of certain 
functionals was provided in \cite{schthaeLIM1, schthaeLIM3}. These papers 
also contain a good intuitive interpretation and a discussion of the 
results. It becomes evident that STIT tessellations show a long-distance 
(in space) dependency behavior which is in clear contrast to 
Poisson-Voronoi tessellations. This causes a quite different limiting 
behavior of STIT compared to Poisson-Voronoi tessellations.

\medskip

In \cite{schthaeLIM1, schthaeLIM3} detailed results for the asymptotic 
variances (in a growing observation window) are proved for the number of 
vertices of planar stationary and isotropic STIT tessellations and for the 
total edge length (planar case) or the total surface area/volume of cell 
boundaries (dimension $\geq 2$), respectively. This work is mainly based 
on an application of martingale theory.

\medskip

In the present paper we apply our results from \cite{martnag15} on the 
$\beta$-mixing coefficient of STIT tessellations in two ways. Consider a 
translation invariant functional $X$ defined on tessellations. The 
convergence is studied of values of $X$ on a growing sequence of bounded 
observation windows, divided by the respective volumes of these windows. 
Due to the ergodicity of homogeneous STIT tessellations the limit should 
be the expectation of $X$ for a window of unit volume.

\medskip

First, for additive functionals $X$ we use the Yoshihara-Heinrich method derived in 
\cite{hein94} (referring back to \cite{yoshi}) in order to find a 
universal upper bound for the variance of the $X$-value for a given 
window. This implies a proposition on the speed of $L^2$-convergence of 
the rescaled $X$-values for a growing sequence of windows.

\medskip

Furthermore, we apply results from ergodic theory 
(Theorem by Akcoglu-Krengel and 
by Smythe, see \cite{kreng}) for multi-parameter subadditive processes to 
homogeneous STIT tessellations. This yields a.s. convergence and $L^1$ 
convergence even if $X$ is a subadditive functional.

\medskip

In Section \ref{sectionSTIT} we give a short description of the STIT 
tessellation process, as a process in time with values in the set of all 
tessellations of the $\ell$ dimensional Euclidean space $\RR^\ell$. Then, 
in Section \ref{betaSTIT}, we recall the definition of the $\beta$-mixing 
coefficient in general and for random tessellations in particular, and we 
quote our result from \cite{martnag15} for STIT tessellations. The main 
result concerning the variance of additive functionals $X$ is then given 
in Section \ref{mainres} together with some examples. The theorem is 
proved in Section \ref{sec:yoshihein}. Finally, Section \ref{sec:ergod} 
provides the application of ergodic theorems to subadditive functionals 
$X$ for STIT.

\section{STIT tessellations}
\label{sectionSTIT}

For the first time, STIT tessellations were defined in \cite{nw}. There, a 
construction of STIT in bounded windows was described in all detail. An 
alternative but equivalent construction was given in \cite{martnag}. 
Throughout the whole paper, we consider tessellations of the 
$\ell$-dimensional Euclidean space $\RR^\ell$ with dimension $\ell \geq 
2$. (Note that the STIT tessellation in $\RR^1$ is generated by a 
stationary Poisson point process on the real axis.)

\medskip

A detailed and sound definition of the measurable space of tessellations of 
a Euclidean space is given in \cite{sw} (Ch. 10, Random Mosaics). A 
tessellation is a set $y$ of polytopes (the cells) with disjoint interiors 
and covering the Euclidean space, where each bounded subset of $\RR^\ell$ 
is intersected by only finitely many cells (locally finiteness condition). 
On the other hand, a tessellation can as well be considered as a closed 
set $\partial y \subset \RR^\ell$ which is the union of the cell boundaries. There is an 
obvious one-to-one relation between both ways of description of a 
tessellation, and their measurable structures can be related 
appropriately, see \cite{sw, martnag}. Denote by $\T$ the set of all 
tessellations of $\RR^\ell$.

\medskip

Let ${\cal C}$ be the set of all compact subsets of 
$\RR^\ell$. We endow $\T$ with the Borel $\sigma$-algebra ${\cal B}(\T)$
of the Fell topology (also known as the topology of closed convergence), 
namely 
$$ 
{\cal B}(\T)=\sigma \left( \{ \{ y\in 
\T :\, \partial y \cap C=\emptyset \} :\, C\in {\cal C} \} \right)
$$
($\sigma({\I})$ is the smallest 
$\sigma-$algebra containing the class $\I$ of sets.)

\medskip

We can also consider tessellations of a bounded window $W$ which is assumed to be a polytope (i.e. 
the convex hull of a finite set of points) with nonempty interior. Denote the set of all those tessellations by $\T\wedge W$. If
$y\in \T$ we denote by $y\land W$ the induced tessellation on $W$.
Its boundary is defined by 
$\partial (y\wedge W)= (\partial y \cap W)\cup {\partial W}$.

\medskip

\medskip

For a window $W$ we introduce the following sub-$\sigma$-algebras of
${\cal B}(\T)$:
\begin{eqnarray}
\nonumber
&{}& {\cal B}(\T_{W})=\sigma \left( \{ \{y\in \T :\, \partial y\cap C=
\emptyset \} :\, C\subseteq W,\, C\in {\cal C} \} \right) 
\hbox{ and }\\
\label{defin}
&{}& {\cal B}(\T_{W^c})=\sigma \left( \{ \{ y\in \T :\, \partial y\cap C=
\emptyset  \} :\, C\subset W^c,\, C\in {\cal C}  \} \right) .
\end{eqnarray}
We notice that if $W'\subseteq W$ then
${\cal B}(\T_{W'})\subseteq {\cal B}(\T_{W})$
and ${\cal B}(\T_{{W}^c})\subseteq {\cal B}(\T_{W'^c})$.

\medskip

We will consider here the action of (spatial) translations on the 
space of tessellations $\T$. For $a\in \RR^\ell$, the translated
tessellation $y+a$ is defined by
$\partial (y+a)=(\partial y)+a=\{z+a: z\in \partial y\}$.

\medskip

\subsection{Construction of STIT}
\label{subsec:construction}

Now we describe a construction of the STIT process, 
which is a particular  
continuous time  Markov process on $\T$.
Let $(\hH , {\mathfrak H})$ denote the measurable space of all hyperplanes 
(i.e. $(\ell -1)$-dimensional affine subspaces) 
in $\RR^\ell$ (cf. \cite{sw})
and $\Lambda$ be a (non-zero) measure
on this space of hyperplanes that satisfies:

\medskip

\noindent {\bf Assumption I:} {\em
The measure $\Lambda$ on $(\hH , {\mathfrak H})$ 
is translation invariant and
possesses the following locally finiteness property:
$\Lambda([B])<\infty$ for all polytopes $B\subset \RR^\ell$, 
where we denote 
$$
[B]=\{H\in {\cal H}: H\cap B\neq \emptyset \}.
$$
Moreover, the support of $\Lambda$ is such that there is no line in
$\RR^\ell$ with the property that all hyperplanes
of the support are parallel to it} (in order
to obtain a.s. bounded cells in the constructed tessellation, cf.
\cite{sw}, Theorem 10.3.2). 

\medskip

This assumption implies  
$0<\Lambda([W])<\infty$ for every window $W$. 
We denote by $\Lambda_{[W]}$ the restriction of $\Lambda$ to 
$[W]$ and by $\hLambda_{[W]}=\Lambda ([W])^{-1}\Lambda_{[W]}$ 
the normalized probability measure.

\medskip

Let us give a brief description of the construction of 
$Y\wedge W=(Y_t\wedge W: t\geq 0)$, 
the random process of STIT tessellations on $W$. 
Let $(\tau_n, n\in \NN)$ be a sequence of independent identically 
distributed random variables, each one of them exponentially 
distributed with parameter $1$.

\medskip

\noindent {\bf Algorithm.}

\medskip

\noindent $(a)$ We define $Y_0\wedge W= \{W\}$ the 
trivial tessellation for the window $W$, and its unique cell is
denoted by $C^1=W$. 

\medskip

\noindent $(b)$ Any cell $C^i$ which is generated in the course of the 
construction has the lifetime $\Lambda ([C^i])^{-1}\, \tau_i$, i.e. an 
exponentially distributed lifetime with parameter $\Lambda ([C^i])$. At 
the end of its lifetime the cell $C^i$ it is divided into two 
new cells by a random hyperplane $H_i$ with law $\hLambda_{[C^i]}=\Lambda 
([C^i])^{-1}\Lambda_{[C^i]}$, where $H_i$ is (if $C^i$ is given) 
conditionally independent from all $\tau_j$ and all $H_j$, $j<i$.

\medskip

\noindent $(c)$ This procedure is 
performed for any extant cell independently. 
It is easy to see that at any time a.s. at most one cell
dies and so a.s. at most only two cells are born.
We enumerate the cells according to the time $t$ when they are 
born: if at this time
there have already born $n$ cells, then we denote the cells 
resulted from a division at $t$ by $C^{n+1}$ and $C^{n+2}$.

\medskip

With this notation, at each time $t>0$ the tessellation $Y_t\land W$ is
constituted by the cells $C^i$ which 'live' at time $t$. In \cite{nw} 
it was shown that there is no explosion, so
at each time $t>0$ the number of cells
of $(Y\land W)_t$, denoted by $\xi_t$, is finite a.s..

\medskip

By construction, the holding time at $t$
$$
\sigma_t=\inf\{s>0: (Y\wedge W)_{t+s}\neq (Y\wedge W)_t\}.
$$
is exponentially distributed with parameter
\begin{equation}
\label{defzeta}
\zeta((Y\wedge W)_t) \hbox{ where } \zeta(y)= \sum_{C\in y} \Lambda([C])
\; \hbox{ for } y\in \T\wedge W.
\end{equation} 

On every window $W$ there exists $Y\land W =((Y\land W)_t: t > 0)$, 
which we call a STIT tessellation process. 
It turns out to be a pure jump Markov process 
and hence it has the strong Markov property. 
Furthermore, if $W'$ and $W$ are windows such that 
$W'\subseteq W$, then for any $t>0$: 
$(Y\wedge W)_t \wedge W' \stackrel{D}{=} (Y\wedge W')_t$, 
where $\stackrel{D}{=}$ denotes the identity of distributions 
(for a proof see \cite{nw}). This yields the existence of a STIT 
tessellation $(Y_t)$ of $\RR^\ell$ such that for all windows $W$ 
we have $Y_t \wedge W\stackrel{D}{=} (Y\wedge W)_t$.

\medskip

A global construction for a STIT process was provided in \cite{mnw}
which shows straightforwardly that $Y=(Y_t: t>0)$ is a Markov process,
so for all $t>0$, $Y_t$ takes values in $\T$. A property we shall
use is that $Y_t$ is spatially stationary (also referred to as homogeneous)
i.e. its law is translation invariant,
$Y_t\stackrel{D}{=}Y_t+a$ for all $a\in \RR^\ell$.
We will denote by $P_t$ the law of $Y_t$ on $(\T, {\cal B}(\T))$.

\section{An upper bound for the spatial $\beta$-mixing coefficient of 
STIT tessellations}
\label{betaSTIT}

\subsection{The $\beta$-mixing coefficient}
\label{Sub1.1}

First, we recall the general definition of the $\beta$-mixing coefficient 
for pairs of $\sigma$-algebras as given in \cite{volroz59}, see also \cite{hein94}. For a 
probability space $(\Omega, {\mathcal D},\bP)$ and sub-$\sigma$-algebras $ 
{\mathcal A},{\mathcal B}\subseteq {\mathcal D}$ denote by $\bP_{\mathcal 
A}$, $\bP_{\mathcal B}$ the restrictions of $\bP$ to $ {\mathcal A}$ and 
${\mathcal B}$ respectively, and by $\bP_{{\mathcal A}\otimes {\mathcal 
B}}$ the restriction to ${\mathcal A}\otimes {\mathcal B}$ of the image measure 
of $\bP$ on $\Omega \times \Omega$ induced by the diagonal mapping $\omega 
\mapsto (\omega ,\omega ) $. Also, and as usual, we denote by 
$\bP_{\mathcal A} \otimes \bP_{\mathcal B}$ the product measure of
$\bP_{\mathcal A}$ and $\bP_{\mathcal B}$. Then, the 
$\beta$-mixing coefficient can be 
given by one of the following expressions:

\begin{eqnarray*}
\label{defbeta}
\beta ({\mathcal A},{\mathcal B}) & = & 
\EE \left(\sup_{B\in {\mathcal B}} \left| \bP(B| {\mathcal A})- 
\bP (B) \right|\right) 
\\ 
 & = & \sup_{C\in {\mathcal A}\otimes {\mathcal B}} \left|  
\bP_{{\mathcal A}\otimes {\mathcal B}}(C)- (\bP_{\mathcal A} 
\otimes \bP_{\mathcal B})(C)  \right| \\
 & = & \frac{1}{2}\;  \sup_{({\overline {A}}, {\overline {B}})} 
\;\; \sum_{r=1}^I \sum_{s=1}^J |
\bP(A_r \cap B_s)-\bP(A_r) \bP( B_s) | ,
\end{eqnarray*}
where in the last expression the supremum is taken over all pairs of finite 
partitions of $\Omega$:
${\overline {A}}= \{{A}_r, r=1,\ldots I \}$ and 
${\overline {B}}=\{ B_s, s=1,\ldots J \}$ with $I,J\in \NN$, 
for events $A_r\in {\cal A}$, $B_s\in {\cal B}$.

\medskip

Now we consider the special case of pairs of $\sigma$-algebras (\ref{defin}) which are 
defined with respect to pairs of windows. This is tailored for dealing 
with models in stochastic geometry (see the seminal paper \cite{hein94} 
and also \cite{dvj}).

\medskip

In \cite{hein94} the $\beta$-mixing rate for a random tessellation with 
distribution $\bP$  on $(\T, {\cal B}(\T))$ is introduced  as follows.
Consider a pair of windows $W'=[-a,a]^\ell \subset W= [-b,b]^\ell$, $0<a<b$. 
Then define
\begin{equation}
\label{defbeta}
\beta (a,b) =  \beta \left( {\cal B}(\T_{W'}) , {\cal B}(\T_{{W}^c}) 
\right), 
\end{equation}
where ${\cal B}(\T_{W'})$ and ${\cal B}(\T_{{W}^c})$ are given by 
(\ref{defin}).

\medskip

In \cite{martnag15} we proved an upper bound for $\beta (a,b)$ for the 
STIT tessellation $Y_t$. To supply it we will introduce the whole 
context.
For a fixed time $t>0$ we choose $\bP=P_t$ on $(\T, {\cal B}(\T))$, 
the marginal distribution of $Y_t$, to study $\beta (a,b)$. 

\medskip

Consider the windows $W'=[-a,a]^\ell$, $W=[-b,b]^\ell$ with $0<a<b$ and 
denote their $(\ell -1)$-dimensional facets by $f_r'$ and $f_r$ 
respectively, $r=1,\ldots ,2\ell$. We define them for $r=1 ,\ldots ,\ell$ 
as $$ f_r'= [-a,a]\times \ldots \times [-a,a]\times \{ a\} \times 
[-a,a]\times \ldots \times [-a,a] $$ with the singleton $\{ a\}$ standing 
at the $r$-th position, and $f_{r+\ell}'= -f_r'$ for $r=1 ,\ldots ,\ell$. 
The $f_r$ are defined as the $f'_r$ respectively, by replacing $a$ by $b$. 
Since $f'_r$ and $f_r$ are disjoint closed convex sets, the 
class of their separating hyperplanes $G_r(a,b)=[f'_r | f_r]$ is 
nonempty and it is a measurable set, that belongs to ${\mathfrak H}$. 
We need the following additional assumption on $\Lambda$ on the sets of
separating hyperplanes $G_r(a,b)=[f'_r | f_r]$, 
$r=1,\ldots, 2\ell$. . 

\medskip

\noindent {\bf Assumption II:} For all $0<a<b$ and the windows 
$W'=[-a,a]^\ell$, $W=[-b,b]^\ell$ we assume that $\Lambda (G_r(a,b))>0$ 
for all $r=1,\ldots , 2\ell$.

\medskip

In \cite{martnag15} we proved (in Theorem $5.3$) the following property on the
$\beta-$mixing coefficient for the STIT tessellations.

\begin{theorem}
\label{decaypoly}
Let $(Y_t)$ be the STIT tessellation determined by the hyperplane measure 
$\Lambda$ satisfying Assumptions I and II. Then for all fixed $t>0$, and  
$\bP=P_t$ being the distribution of $Y_t$ on $(\T, {\cal B}(\T))$, it holds 
$\lim_{b\to \infty} \beta (a,b) =0$ for all $a>0$.
Moreover,{defin} for all $\theta\in (0,1)$ 
there exists a constant $\chi=\chi(t,a,\theta)<\infty$
such that $\beta (a,b)\le \chi \, b^{-\theta}$ for all $b> a>0$. $\Box$
\end{theorem}

\section{Main result}
\label{mainres}

Using our results for an upper bound of the $\beta$-mixing coefficient, we 
obtain a proposition concerning $L^2$-convergence including a bound for 
the speed of convergence.

\medskip

For $c=(c_1,\ldots , c_\ell ) ,\, d=(d_1,\ldots , d_\ell )\in \RR^\ell$ let
$[c,d[=\displaystyle{\prod_{r=1}^\ell} [c_r,d_r[$. We denote by
${\mathcal V}=\{ [c,d[ \, :\ c, d\in \RR^\ell, [c,d[ \neq \emptyset \}$
the family of nonempty half-open cuboids.

\medskip
 
For $y\in \T$ and $V\in {\mathcal V}$, by $y\wedge V$ we mean the 
tessellation induced by $y$ in $V$. 

\medskip

\begin{theorem}
\label{main1}
Let $Y_t$ be the STIT tessellation at time $t>0$ determined by the 
hyperplane measure $\Lambda$ satisfying Assumptions I and II. 
Further, for any $V\in {\mathcal V}$ 
let $X(V , \cdot ):\T \wedge V \to \RR$ 
be a functional with the following 
properties (briefly we write $X(V,y)$ for $X(V,y\wedge V)$):
\begin{enumerate}

\item If $n\in \NN$ and $V_1,\ldots , V_n \in {\mathcal V}$ are 
pairwise disjoint  such that $\bigcup_{r=1}^n V_r \in {\mathcal V}$, then
\begin{equation}
\label{eq:additiv}
X\left(\bigcup_{r=1}^n V_r , \cdot \right) = \sum_{r=1}^n X(V_r,\cdot ) 
\qquad (additivity).
\end{equation}

\item $X(V,y) = X(V+i,y+i)$ for all $V \in {\mathcal V}$, $i\in \RR^\ell$, 
$y \in \T$.

\item For some $\delta>0$ and some $V\in {\mathcal V}$ we have 
$\EE \left(X(V, Y_t)^{2+\delta} \right) <\infty$.  
\end{enumerate}
Then  for all  $\theta \in (0,1)$
\begin{equation}
\label{asymvar}
 Var\left(\frac{1}{(2n)^\ell} X([-n,n[^\ell , Y_t)\right) 
\leq O\left(n^{-\theta \frac{\delta}{2+\delta}} \right) \quad  
 \;  n\to \infty.
\end{equation}
\end{theorem}

We note that inequality (\ref{upperbound}) which we will supply in the 
proof of this theorem, will provide an upper bound for the variance for 
any fixed $n\in \NN$.

\medskip

Since the law of $Y_t$ is translation invariant we have 
$Y_t\wedge W\stackrel{D}{=}[(Y_t+a)\wedge (W+a)]-a$ for all 
$a\in \RR^\ell$. 
Hence, 
$$
\EE\left(X([-n,n[^\ell , Y_t)\right)=(2n)^\ell
\EE\left(X([0,1[^\ell, Y_t)\right),
$$
and the variance can be written as
$$
Var\left(\frac{1}{(2n)^\ell} X([-n,n[^\ell , Y_t)\right) =
\EE  \left(\left(\frac{1}{(2n)^\ell} X([-n,n[^\ell, Y_t) - 
\EE X([0,1[^\ell , Y_t) \right)^2 \right).
$$
Therefore relation (\ref{asymvar}) immediately indicates $L^2$-convergence
in the law of large numbers. The expectation 
$\EE\left(X([0,1[^\ell, Y_t)\right)$ is  called 
the $X$-density of $Y_t$.

\medskip

Let us compare our theorem with the results by Schreiber and Th\"ale. 
The papers \cite{schthaeLIM1, schthaeSecondOrder12} contain 
a rich material concerning the second-order properties of STIT tessellations. 
Here we only refer to some asymptotic results that can be compared 
immediately with our proposition.

\medskip

Let $Y_t$ be a spatially stationary {\em and isotropic} STIT tessellation 
at time $t>0$.

\noindent $(i)$ Let $\ell =2$.
If the random field $X([-n,n[^2 , Y_t)$ is the number of vertices, or
$X([-n,n[^2, Y_t)$ is the number of center points of maximal
($I-$)segments, or $X([-n,n[^2, Y_t)$ is the total length of edges, 
of $Y_t$ in $[-n,n[^2$, then: 
$$
Var\left(\frac{1}{(2n)^2} X([-n,n[^2, Y_t)\right) = 
O\left(n^{-2} \ln n \right) \quad  \mbox{ for  }  \ n\to \infty ,
$$ 
(see \cite{schthaeLIM1}, Corollary 2, 
and \cite{schthaeSecondOrder12}, Theorem 6.1).
\medskip

\noindent $(ii)$ Let $\ell \geq 3$. If
$X([-n,n[^\ell, Y_t)$ is the total surface area of cell boundaries of $Y_t$ 
in $[-n,n[^\ell$ then: 
$$
Var\left(\frac{1}{(2n)^\ell} X([-n,n[^\ell, Y_t)\right) = 
O\left(n^{-2} \right) \quad  \mbox{ for  }  \ n\to \infty ,
$$ 
(see \cite{schthaeSecondOrder12}, Theorem 6.1).
Hence, in the above particular cases the asymptotic variance 
is considerably smaller than our upper bound (\ref{asymvar}).

\medskip

\noindent Some examples where Theorem \ref{main1} can be applied and up to 
now no better bound for the variance is known, are the following ones.

\begin{enumerate}
\item As in (i) and (ii) above, but anisotropic STIT tessellations.

\item $(1)$ $\ell \geq  3$, $0\leq k\leq \ell -1$: 
\begin{enumerate}
\item total $k$-volume of $k$-dimensional faces of cells in $[-n,n[^\ell$, 

\item number of reference points of 
$k$-dimensional faces of cells in $[-n,n[^\ell$,

\item number of reference
points of other $k$-dimensional objects of $Y_t$ in $[-n,n[^\ell$
(see \cite{weisscow11, thaeweiss13}).
\end{enumerate}
\end{enumerate}
Reference points are uniquely determined points assigned to all the 
considered sets by the same rule, e.g. their circumcenters or their 'most 
left' or their 'lowest' points.

\section{The Yoshihara-Heinrich method for upper bounds for the variance}
\label{sec:yoshihein}

\subsection{Upper bound for the covariance}
We will obtain the upper bound (\ref{asymvar}) by using the following result, that can 
be found in \cite{hein94}, formula (4.6).

\begin{theorem} (Yoshihara-Heinrich) 
\label{th:yohei}
Let $(\Omega, {\mathcal D},\bP)$ be a probability space, 
${\mathcal A},{\mathcal B}\subseteq {\mathcal D}$ sub-$\sigma$-algebras 
and $\beta ({\mathcal A},{\mathcal B})$ the respective 
$\beta$-mixing (absolute-regularity) coefficient.
Then for any ${\mathcal A}\otimes {\mathcal B}$-measurable 
function $h: \Omega \times  \Omega \to \RR^d$, $d\geq 1$, and 
all $\delta > 0$
\begin{eqnarray}
\label{yohei}
&{}& \int |h| \, |{\rm d} (\bP_{{\mathcal A}\otimes 
{\mathcal B}} -\bP_{\mathcal A} \otimes\bP_{\mathcal B})|  \\
&\leq & 2 \max \left\{  \left( \int |h|^{1+\delta} \, {\rm d}
\bP_{{\mathcal A}\otimes {\mathcal B}}\right)^{\frac{1}{1+\delta}},
\left( \int |h|^{1+\delta} \, {\rm d}\bP_{\mathcal A} \otimes 
\bP_{\mathcal B}\right)^{\frac{1}{1+\delta}}  \right\} \, 
(\beta ({\mathcal A},{\mathcal B}))^{\frac{\delta}{1+\delta}} 
\nonumber
\end{eqnarray}
\end{theorem}

We note that in the case $\beta({\mathcal A},{\mathcal B})=0$
both sides vanish, and when $\beta({\mathcal A},{\mathcal B})>0$
and $\int |h|^{1+\delta} \, {\rm d}\bP_{\mathcal A} \otimes 
\bP_{\mathcal B}=\infty$ the right hand side is $\infty$.

\medskip

Choosing the function $h$ as  
$h(\omega_1 , \omega_2) = X(\omega_1)\cdot  Z(\omega_2)$, 
$(\omega_1 , \omega_2)\in \Omega \times \Omega$, for 
real-valued random variables $X,\,  Z$ 
(with finite second moments) yields the following upper 
bound for their covariance.

\begin{corollary}
\label{covyohei}
For all real valued random variables $X,  Z\in L^2 (P)$ and all $\delta >0$
\begin{equation*}
|\hbox{Cov}(X, Z) | \leq 2 
\left(\EE\left(|X|^{2+\delta} \right)\right)^{\frac{1}{2+\delta}} 
\left(\EE \left(| Z|^{2+\delta}\right)\right)^{\frac{1}{2+\delta}}
\left(\beta(\sigma (X),\sigma ( Z))\right)^{\frac{\delta}{2+\delta}} ,
\end{equation*}
where $\sigma (X),\  \sigma ( Z)$ denote the $\sigma$-algebras generated 
by $X$ and $ Z$ respectively.
\end{corollary}

\begin{proof}
 Theorem \ref{th:yohei} yields for $\theta >0$
\begin{eqnarray*}
&{}& | \hbox{Cov} (X, Z) | \\
& = & \left| \int X\cdot  Z {\rm d}\bP_{\sigma (X)\otimes \sigma ( Z)} - 
\int X\cdot  Z {\rm d}\bP_{\sigma (X)} \otimes \bP_{\sigma(Z)} \right| 
\\
& \leq & \int | X\cdot  Z | \, \left| {\rm d} 
\left(\bP_{\sigma (X)\otimes \sigma ( Z)} -\bP_{\sigma (X)} \otimes 
\bP_{\sigma ( Z)} \right) \right| \\
&\leq & 2 \max \left\{\left( \int \! |X Z|^{1+\theta} \, {\rm d} 
\bP_{\sigma (X)\otimes \sigma ( Z)}\! \right)^{\frac{1}{1+\theta}} \!\!,
\left( \int \! |X Z|^{1+\theta} \, {\rm d}\bP_{\sigma (X)} \otimes 
\bP_{\sigma ( Z)}\!\right)^{\frac{1}{1+\theta}}   \right\} \\ 
&{}& \;\, \cdot \, (\beta (\sigma (X),\sigma ( 
Z)))^{\frac{\theta}{1+\theta}} 
\end{eqnarray*}

Applying the H\"older inequality we obtain 
\begin{eqnarray*}
&{}&\left(\int |X|^{1+\theta}\, | Z|^{1+\theta} \, 
{\rm d}\bP_{\sigma (X)\otimes \sigma ( Z)}\right)^{\frac{1}{1+\theta}} \\
&{}& \, \leq \left( \int |X|^{2+2\theta}{\rm d}\bP_{\sigma (X)
\otimes \sigma ( Z)} \right)^{\frac{1}{2(1+\theta)}} 
\left( \int | Z|^{2+2\theta}{\rm d}\bP_{\sigma(X)
\otimes \sigma ( Z)} \right)^{\frac{1}{2(1+\theta)}}.
\end{eqnarray*}
On the other hand, the Jensen inequality yields,
\begin{eqnarray*}
\left( \int \! |X Z|^{1+\theta} \, {\rm d}\bP_{\sigma (X)} \otimes
\bP_{\sigma ( Z)}\right)^{\frac{1}{1+\theta}} &=&
\left(\int |X|^{1+\theta} \, {\rm d}\bP_{\sigma (X)} 
\int | Z|^{1+\theta} 
\, {\rm d}\bP_{\sigma( Z)}\right)^{\frac{1}{1+\theta}}\\ 
&\leq &  
\left(\int |X|^{2+2\theta}{\rm d}\bP_{\sigma (X)} 
\right)^{\frac{1}{2(1+\theta)}}
\left(\int | Z|^{2+2\theta}{\rm d}\bP_{\sigma ( Z)} 
\right)^{\frac{1}{2(1+\theta)}}.
\end{eqnarray*}
Now choose $\delta= 2 \theta$, and note that
$$
\EE\left(|X|^{2+\delta}\right)=\int |X|^{2+\delta}{\rm d}\bP_{\sigma(X)}
\hbox{ and }
\EE\left(| Z|^{2+\delta}\right)=\int |X|^{2+\delta}{\rm d}\bP_{\sigma(X)}\,,
$$
 and the result is shown.
\end{proof}

In the next Section the results for the covariance are applied to provide 
an upper bound for the variance of additive functionals in large windows.

\subsection{Proof of Theorem  \ref{main1}} 

Consider the half-open cube $W^o_n=[-n,n[^\ell$, $n\in \NN$, and partition 
it into half-open unit cubes
$$
c_i := \prod_{r=1}^\ell [i_r, i_r+1[ \hbox{ for } 
i=(i_1,\ldots ,i_\ell) \in \ZZ ^\ell \cap [-n,n[^\ell \,.
$$
Now, if $X$ is an additive functional  then 
$X(W^o_n, Y_t) = \sum_i X(c_i, Y_t)$. For 
simplicity we define $X(c_i, Y_t)=0$ 
if $i\not\in \ZZ^\ell \cap [-n,n[^\ell$, and we write 
$X_i$ for $X(c_i, Y_t)$.

\medskip

On $\ZZ^\ell$ we use the maximum metric 
$$
d(i,j)=\max \{ |i_r - j_r| :\, 
r=1,\ldots , \ell \}, \;  i,  j\in \ZZ^\ell. 
$$
Note that the number of points with integer coordinates in the ball of 
radius $k>0$ with center $i\in \ZZ^\ell$ is $\big|\{ j \in \ZZ^\ell : 
d(i,j)\le k\}\big|=(2k+1)^\ell$ and so
$$
\big|\{ j \in \ZZ^\ell: d( i, j)= k\}\big|=(2k+1)^\ell-(2k-1)^\ell.
$$

Note that if $d( i, j)=1$ then $c_i$ and $c_j$ their 
closures have a non-empty intersection. On the other hand
given two cubes $c_i, c_j$ with 
$d( i, j)>1$ we can shift them simultaneously such that the center 
of $c_i$ coincides with the origin and closure of the the shifted $c_i$ can be 
identified with $W_a$, $a=\frac{1}{2}$, while the shifted $c_j$ is 
located in the complement of $W_b= [-b,b]^\ell$, $b=d( i, j)-\kappa$ for all $0<\kappa < \frac{1}{2}$. 
(This notation corresponds to that one which we used in \cite{martnag15}.) 

\medskip

We apply Corollary \ref{covyohei} to all these pairs $c_i, c_j$ with 
$d( i, j)>1$ and obtain for arbitrary $\delta >0$, using (\ref{defbeta}) and the translation invariance of $X$,
\begin{equation}
\label{boundcov}
|\hbox{Cov} (X_i,X_j)| \leq 2 
\left(\EE |X(W_\frac{1}{2}, Y_t)|^{2+\delta} \right)^{\frac{2}{2+\delta}} 
\left(\beta \left( \textstyle{\frac{1}{2}}, d( i, j)- \kappa  \right) 
\right)^{\frac{\delta}{2+\delta}}
\end{equation}
For pairs $c_i, c_j$ with $d( i, j)=1$ we use the inequality
$$
|\hbox{Cov}(X_i, X_j)| \leq \sqrt{Var(X_i) \cdot Var(X_j)} = Var(X_i)=
Var(X(W_\frac{1}{2},Y_t)) .
$$ 

By writing $\sum_{ i}\,$ for the summation 
over all $ i\in \ZZ^\ell$ (which in our setting is equivalent 
to summation over all $ i\in \ZZ^\ell \cap [-n,n[^\ell$) 
\begin{eqnarray*}
&{}& Var(X(W^o_n, Y_t)) = Var \left(\sum_{ i} X_i \right) \\
&=& \sum_{ i} Var(X_i)  + 
\sum_{ i} \sum_{ j:d( i, j)=1} \!\!\! \hbox{Cov}(X_i,X_j) 
+ \sum_{ i} \sum_{k=2}^{2n-1} \sum_{ j:d( i, j)=k} \!\!\! 
\hbox{Cov}(X_i,X_j) \\
&\leq & \sum_{ i} Var(X_i)  
+ \sum_{ i} \sum_{j:d( i, j)=1} \!\!\! Var(X_i) 
+ \sum_{ i} \sum_{k=2}^{2n-1} \sum_{j:d( i, j)=k} \!\!\!
2 \left(\EE |X_i|^{2+\delta} 
\right)^{\frac{2}{2+\delta}} 
\left(\beta \left( \textstyle{\frac{1}{2}}, 
k-\kappa  \right) \right)^{\frac{\delta}{2+\delta}} \\
& \leq & (2n)^\ell Var(X_1) 
+ (2n)^\ell \, (3^{\ell} -1) Var(X_1) \\ 
&{}& \; + 2 
\left(\EE \left(|X_1|^{2+\delta} \right)\right)^{\frac{2}{2+\delta}} 
(2n)^\ell \, \sum_{k=2}^{2n-1} ((2k+1)^\ell - (2k-1)^\ell )  
\left(\beta \left( \textstyle{\frac{1}{2}}, 
k- \kappa  \right) \right)^{\frac{\delta}{2+\delta}}\\
& \leq & (2n)^\ell \left[ 
3^{\ell} Var(X_1)  + 
2 \left(\EE\left(|X_1|^{2+\delta} \right)\right)^{\frac{2}{2+\delta}} 
 2^{2\ell-1} \, 
\sum_{k=2}^{2n-1}  k^{\ell-1} \,  
\left(\beta \left( \textstyle{\frac{1}{2}}, k- 
\kappa  \right) \right)^{\frac{\delta}{2+\delta}}  \right]
\end{eqnarray*}   
Hence, for $X(W^o_n,Y_t)$ divided by the volume of $W_n$ and using 
the upper bound for $\beta$ as given in Theorem 5.3 in \cite{martnag15}, 
we obtain for any $\theta\in (0,1)$, there is a constant 
$\chi(\theta)$ such that
\begin{eqnarray*}
& {} & Var\left(\frac{1}{(2n)^\ell} X(W^o_n,Y_t)\right) \\
& \leq &  
\frac{3^{\ell}}{(2n)^\ell} Var(X_1) +  
 2^{\ell}  
\left( \EE |X_1|^{2+\delta} \right)^{\frac{2}{2+\delta}}
 \frac{1}{(2n)^\ell}  \sum_{k=2}^{2n-1}  k^{\ell -1} 
\left(\beta \left( \textstyle{\frac{1}{2}}, 
k- \kappa  \right) \right)^{\frac{\delta}{2+\delta}} \\
&\leq & 
\frac{3^{\ell}}{(2n)^\ell}  Var(X_i) +  
 2^{\ell}  
\left( \EE |X_1|^{2+\delta} \right)^{\frac{2}{2+\delta}}
 \frac{1}{(2n)^\ell} \chi (\theta) 
\sum_{k=2}^{2n-1}  k^{\ell -1} 
\left(k-\textstyle{\kappa} \right) ^{-\frac{\theta \delta}{2+\delta}} 
.\\
\end{eqnarray*}

Let us consider the asymptotic behavior for $n\to \infty$ of
$$
\frac{1}{(2n)^\ell}  \sum_{k=2}^{2n-1}  k^{\ell -1} 
\left(k-\textstyle{\kappa } \right)^{-\frac{\theta\delta}{2+\delta}} .
$$

Let us denote $\frac{{\theta}\delta}{2+\delta} = 1-\rho$. Note that 
$k^{\ell -1} \left(k-\textstyle{\kappa } \right)^{-(1-\rho)}$ 
is increasing in $k$  if $\ell \geq 2$. Thus
$$
\frac{1}{(2n)^\ell}  \sum_{k=2}^{2n-1}  
k^{\ell -1} \left(k-\textstyle{\kappa } \right) ^{-(1-\rho)} 
\leq \frac{1}{(2n)^\ell} 2n (2n)^{\ell-1} 
\left(2n-\textstyle{\kappa } \right)^{-(1-\rho)}
= \left(2n-\textstyle{\kappa } \right)^{-(1-\rho)}.
$$

Hence, we have
\begin{equation}
\label{upperbound}
Var\left(\frac{1}{(2n)^\ell} X(W_n,Y_t)\right) \leq 
\frac{3^{\ell}}{(2n)^\ell}  Var(X_ 1)  +   2^{\ell}  
\left( \EE |X_1|^{2+\delta} \right)^{\frac{2}{2+\delta}} \chi (\theta) 
\left(2n-\textstyle{\kappa } \right)^{-(1-\rho)}
\end{equation}
and thus
$$
 Var\left[\frac{1}{(2n)^\ell} X(W_n,Y_t)\right] 
\leq O\left(n^{-(1-\rho)} \right), \ n\to \infty .
$$
which is the assertion (\ref{asymvar}).


\section{Ergodic theorems for subadditive functionals of random tessellations}
\label{sec:ergod}

If the functional is not additive, we cannot use the method used in 
Section \ref{sec:yoshihein}. But for subadditive functionals, Theorem 2.9 
in \cite{kreng} by Akcoglu-Krengel and Theorem 2.3 in \cite{kreng} by 
Smythe together with the property that STIT tessellations are 
$\beta$-mixing can be applied to formulate an ergodic theorem. Here we do 
not aim at the highest level of generality, but on a formulation similar 
to Theorem \ref{main1}, which appears to be well adapted to applications 
in stochastic geometry.

\begin{theorem}
\label{ergod_krengel}
Let $Y_t$ be the STIT tessellation at time $t>0$ determined by the 
hyperplane measure $\Lambda$ satisfying Assumptions I and II. 
Further, for any $V\in {\mathcal V}$ 
let $X(V , \cdot ):\T \wedge V \to \RR$ 
be a functional with the following 
properties (briefly we write $X(V,y)$ for $X(V,y\wedge V)$):
\begin{enumerate}

\item If $n\in \NN$ and $V_1,\ldots , V_n \in {\mathcal V}$ are 
pairwise disjoint  such that $\bigcup_{r=1}^n V_r \in {\mathcal V}$, then
\begin{equation}
\label{eq:subadditiv}
X\left( \bigcup_{r=1}^n V_r , \cdot \right) \leq 
\sum_{r=1}^n X(V_r,\cdot ) \qquad (subadditivity).
\end{equation}

\item $\;X(V,y) = X(V+i,y+i)$ for all $V \in {\mathcal V}$, $i\in 
\RR^\ell$, 
$y \in \T$.

\item 
$\; \displaystyle{
\gamma = \inf \left\{  \frac{1}{(2n)^\ell} 
\EE X([-n,n[^\ell ,Y_t) , \, n\in \NN \right\} > -\infty .
} $
\end{enumerate}
Then  
\begin{equation}\label{ergod_as}
\lim_{n\to \infty} \frac{1}{(2n)^\ell } X([-n,n[^\ell ,Y_t) =\gamma \quad a.s.
\end{equation}

For the $L^1$-convergence in (\ref{ergod_as}) it is sufficient if 
(\ref{eq:subadditiv}) holds for $n=2$ only (2-subadditivity).

\end{theorem}

Note that for additive $X$ we have $\gamma =  \EE X([0,1[^\ell , Y_t)$.

\medskip

{\bf Comments on the proof:}  
This is a special case of a subadditive process in the sense of Definition 
2.1 in Chapter $6$ in \cite{kreng}; the sequence $(W_n: n\in \NN)$ is 
increasing and therefore a regular family, see Definition 2.4 ibidem. As 
$Y_t$ is $\beta$-mixing, it is ergodic. Thus the theorem is an immediate 
corollary of Theorems 2.9 and 2.3 ibidem.

\medskip

Examples: Besides the additive functionals described in the example in 
Section \ref{mainres}, some subadditive functionals which fulfill the 
assumptions of Theorem \ref{ergod_krengel} are:

\begin{enumerate}

\item Sum of the values of translation invariant nonnegative functionals 
of those cells $C$ of the tessellation which are {\em completely} contained in a 
window $W_n$, i.e. $C\subset W_n$. This can e.g. be the intrinsic volumes 
of these cells, or the diameter (i.e. the maximal breadth) of these cells, 
or the number of $k$-dimensional faces of these cells, $k=0,1,\ldots , 
\ell -1$.

\item As before, but for all cells or parts of cells of a tessellation 
$y_t$ say, which are visible in the window $W_n$, i.e. for all $C\in y_t 
\wedge W_n$.

\item Power of order $0<\alpha <1$ of a nonnegative functional, e.g.  
$(\lambda_{\ell -1}(\partial y_t \wedge W_n))^\alpha $, where 
$\lambda_{\ell -1}$ denotes the ${(\ell -1)}$-dimensional volume.

\end{enumerate}

In all these examples, the functional $X$ is nonnegative, and hence 
$\gamma \geq 0$.

\section{Discussion}

The results in \cite{redenthae2013} indicate that several second-order 
quantities of the isotropic STIT tessellations are between the 
corresponding values for isotropic Poisson hyperplane tessellations on one 
side and for Poisson-Voronoi tessellations on the other side. Thus one can 
use results for these two other models to formulate conjectures concerning 
the variance of functionals for isotropic STIT tessellations.

\vspace{1cm}

{\bf Acknowledgments} 
The authors thank Lothar Heinrich for essential suggestions and comments.
They are indebted for the support of Program Basal CMM 
from CONICYT (Chile) and by DAAD (Germany).


\medskip

\begin{thebibliography}{99}

\smallskip



\smallskip


\bibitem{dvj}
{\sc Daley, D.J. and  Vere-Jones, D.} (2008).
{\em An Introduction to the Theory of Point Processes. 
Vol. II: General Theory and Structure.}  2nd ed., Springer.

\smallskip


\bibitem{hein94} 
{\sc  Heinrich, L.} (1994).
 { Normal approximation for some mean-value estimates of absolutely 
regular tessellations.} {\em Math. Methods Statist.} \textbf{3}, 1--24.

\smallskip

\bibitem{hein}
{\sc Heinrich, L.} (2012).
{ Asymptotic methods in statistics of random point processes.}
in: {\em E. Spodarav (ed.): Lectures on Stochastic Geometry, 
Spatial Statistics and Random Fields. Asymptotic
Methods.} Lecture Notes in Mathematics, Springer,  115--150. 

\smallskip

\bibitem{kreng}
{\sc Krengel, U.} (1985).
{\em Ergodic Theorems.} de Gruyter, Berlin.

\smallskip   

\bibitem{martnag}
{\sc  Mart{\'\i}nez, S. and Nagel, W.} (2012).
{Ergodic description of STIT tessellations}.
{\em Stochastics: An Int. Journ. of Prob. and Stoch. Proc.} 
\textbf{84},  113--134.

\smallskip

\bibitem{martnag15}
{\sc  Mart{\'\i}nez, S. and Nagel, W.} (2016).
{The $\beta$-mixing rate of STIT tessellations}. 
{\em Stochastics: An Int. Journ. of Prob. and Stoch. Proc.} \textbf{88}, 396--414.
\smallskip

\bibitem{mnw} 
{\sc Mecke, J., Nagel, W. and Weiss, V.} (2008). A global 
construction of homogeneous random planar tessellations that are 
stable under iteration. {\em Stochastics: An Int. Journ. of Prob. and Stoch. Proc.} \textbf{80}, 51--67.

\smallskip

\bibitem{nw} 
{\sc Nagel, W. and Weiss, V.} (2005). Crack STIT tessellations: 
Characterization of stationary random tessellations stable with 
respect to iteration. {\em Adv. Appl. Probab.} {\bf 37}, 859--883.

\smallskip

\bibitem{ngzes}
{\sc Nguyen, X. X. and Zessin, H.} (1979) Ergodic theorems for 
spatial processes. {\em Zeitschr. Wahrsch. verw. Geb.} {\bf 48}, 133--158.

\smallskip


\bibitem{redenthae2013}
{\sc Redenbach, C. and Th\"ale, C.} (2011). {Second-order comparison of 
three fundamental tessellation models.} 
{\em Statistics: A Journal of Theoretical and Applied Statistics.} 
\textbf{47}, 237--257.


\smallskip


\bibitem{sw} 
{\sc Schneider, R. and Weil, W.} (2008). 
{\em Stochastic and Integral Geometry.} 
Springer.

\smallskip

\bibitem{schthaeLIM1}
{\sc Schreiber, T. and Th\"ale, C.} (2010). 
{Second-order properties and 
central limit theory for the vertex process of iteration infinitely 
divisible and iteration stable random tessellations in the plane}. 
{\em Adv. Appl. Probab.} \textbf{42}, 913--935.

\smallskip

\bibitem{schthaeSecondOrder12}
{\sc Schreiber, T. and Th\"ale, C.} (2012). 
{Second-order theory for iteration stable tessellations}. 
{\em Probab. Math. Statist.}, \textbf{32}, 281--300.
      

\bibitem{schthaeLIM3}
{\sc Schreiber, T. and Th\"ale, C.} (2013).
{ Limit theorems for iteration stable tessellations}.
{\em The Annals of Probability}, \textbf{41}, 2261--2278.

\smallskip

\bibitem{tempel}
{\sc Tempel'man, A. A.} (1972).
Ergodic theorems for general dynamical systems. 
{\em Trans. Moscow Math. Soc.} {\bf 26}, 94--132.

\smallskip

\bibitem{thaeweiss13}
{\sc Th\"ale, C. and Weiss, V.} (2013).
The combinatorial structure of spatial STIT tessellations.
{\em Discrete \& Combinatorial Geometry} {\bf 50}, 649--672.

\smallskip

\bibitem{volroz59}
{\sc Volkonskii V.A. and Rozanov, Yu.A.} (1959).
Some limit theorems for random functions. I.
{\em Theory Probab. Appl.} {\bf 4}, 178--197.

\smallskip

\bibitem{weisscow11}
{\sc  Weiss, V. and Cowan, R.} (2011).
Topological realationships in spatial tessellations.
{\em Adv. Appl. Probab.} \textbf{43}, 963--984.

\smallskip

\bibitem{yoshi}
{\sc Yoshihara, K.-I.} (1976). Limiting behavior of U-statistics for 
stationary absolutely regular processes. {\em Zeitschr. Wahrsch. verw. Geb.} 
{\bf 35}, 237--252.

\end{thebibliography}
\end{document}